\documentclass{birkjour}
\usepackage{subfigure}
\usepackage{verbatim}
\usepackage{graphicx}
 \newtheorem{thm}{Theorem}[section]
 \newtheorem{corollary}[thm]{Corollary}
 \newtheorem{lemma}[thm]{Lemma}
 \newtheorem{Proposition}[thm]{Proposition}
 \theoremstyle{definition}
 
 \theoremstyle{remark}
 \newtheorem{remark}[thm]{Remark}
 \newtheorem{example}{Example}
 \numberwithin{equation}{section}

\newcommand{\A}{\mathcal{A}}
 \newcommand{\R}{\mathbb{R}}

        \newcommand{\E}{\mathcal{E}}
    \newcommand{\B}{\mathcal{B}}

\begin{document}

%
%

\title[ Affine geometry of polylines ]
 {Affine geometry of equal-volume polygons \newline in $3$-space}

\author[M.Craizer]{Marcos Craizer}

\address{%
Departamento de Matem\'{a}tica- Pontif\'{i}cia Universidade Cat\'{o}lica do Rio de Janeiro\br
Rio de Janeiro-RJ- Brazil
}
\email{craizer@puc-rio.br}

\author[S.Pesco]{Sinesio Pesco}

\address{%
Departamento de Matem\'{a}tica- Pontif\'{i}cia Universidade Cat\'{o}lica do Rio de Janeiro\br
Rio de Janeiro-RJ- Brazil
}
\email{sinesio@puc-rio.br}

\thanks{The first author thanks CNPq for financial support during the preparation of this paper.}

\subjclass{ 53A15, 53A20}

\keywords{Darboux vector field, Affine arc-length parameterization, Affine evolute, Projective length, Discrete affine geometry}

\date{September 27, 2016}

\begin{abstract}
Equal-volume polygons are obtained from adequate discretizations of curves 
in $3$-space, contained or not in surfaces. In this paper we explore the similarities 
of these polygons with the affine arc-length parameterized 
smooth curves to develop a theory of discrete affine invariants.  
Besides obtaining discrete affine invariants, equal-volume polygons can also be used to estimate projective invariants of a planar curve. 
This theory has many potential applications, among them evaluation of the quality 
and computation of affine invariants of silhouette curves. 
\end{abstract}

\maketitle

\section{Introduction}

We say that a smooth curve $\gamma(t)$ in $2$-space is parameterized by {\it affine arc-length} if $[\gamma'(t),\gamma''(t)]=1$, where $[\cdot,\cdot]$ denotes the determinant 
of $2$ vectors. For polygons, the corresponding condition is that the area of the triangle determined by three consecutive vertices is constant.
Planar polygons satisfying  this last condition are called {\it equal-area}
and the affine geometry of these polygons has been recently studied (\cite{Craizer13},\cite{Kaferbock14}). 
In this paper we generalize this study to polygons in $3$-space by considering the concept of equal-volume polygons.
Since we obtain discrete counterparts of known objects of the smooth theory, our results clearly belong to the field of Discrete Differential Geometry.

For a smooth curve $\phi$ contained in a surface $M$, we say that the parameterization $\phi(t)$ is {\it adapted to M} if 
\begin{equation}\label{eq:AffineArcLengthDarboux}
[\phi'(t),\phi''(t),\xi(t)]=1,
\end{equation}
where $[\cdot,\cdot,\cdot]$ denotes the determinant of $3$ vectors and $\xi$ is the parallel Darboux vector field of $\phi\subset M$ (\cite{Craizer15}). In centro-affine geometry, we consider 
curves $\phi$ in $3$-space together with a distinguished origin 
$O$, and we say that $\phi(t)$ is parameterized by centro-affine arc-length with respect to $O$ if 
\begin{equation}\label{eq:CentroAffineArcLength}
[\phi(t)-O,\phi'(t),\phi''(t)]=1,
\end{equation}
(\cite{Giblin}). A smooth curve $\Phi(t)$ in $3$-space is parameterized by {\it affine arc-length} if 
\begin{equation}\label{eq:AffineArcLength}
[\Phi'(t),\Phi''(t),\Phi'''(t)]=1,
\end{equation}
 (\cite{Davis},\cite{Izu2}).   We observe that, in all these contexts, the basic condition is the constancy of some volume. In this paper, 
we describe polygons in $3$-space whose corresponding volumes are constant, and we call them {\it equal-volume} polygons. We obtain affine invariant measures only for these equal-volume
polygons, but we also describe a simple algorithm that, by re-sampling an arbitrary polygon, obtain an equal-volume one.

We say that  a smooth curve $\phi$ contained in a surface $M$ is {\it non-degenerate} if its osculating plane does not coincide with the tangent plane
of $M$ at any point. For such curves, there exists a vector field $\xi$ tangent to $M$ and transversal to $\phi$ such that its derivative $\xi'$ 
is tangent to $M$. The direction defined by $\xi$ is unique and is called the {\it Darboux direction} of $\phi\subset M$. Moreover,
there exists a vector field $\xi$ in the Darboux direction such that $\xi'$ is tangent to the curve $\phi$, i.e., 
\begin{equation}\label{eq:DefineSigma}
\xi'(t)=-\sigma(t)\phi'(t),
\end{equation}
for some scalar function $\sigma$. This vector field 
is unique up to a multiplicative constant and is called the {\it parallel Darboux vector field}. 
It turns out that $\phi\subset M$ is a silhouette curve with respect to some point $O$ if and only if $\sigma$ is constant (\cite{Craizer15}).

As a discrete model for curves contained in surfaces, consider a polyhedron $M$ whose faces are planar quadrilaterals and let 
$\phi(i)$ be vertices of a polygon $\phi$ whose sides are connecting opposite edges of a face of $M$. The edges of $M$ containing 
vertices $\phi(i)$ correspond to Darboux directions and we can choose a vector field $\xi(i)$ in this direction such that 
the difference $\xi(i+1)-\xi(i)=\xi'(i+\tfrac{1}{2})$ is parallel to the corresponding side of the polygon $\phi$. This vector field is unique
up to a multiplicative constant, and is called the {\it parallel Darboux vector field}. We can write
a discrete counterpart of equation \eqref{eq:DefineSigma}, namely
\begin{equation}\label{eq:DiscSigma}
\xi'(i+\tfrac{1}{2})=-\sigma(i+\tfrac{1}{2})\phi'(i+\tfrac{1}{2}),
\end{equation}
for some scalar function $\sigma$, where we are replacing derivatives by differences. 
We prove that $\phi$ is a silhouette polygon for the polyhedron $M$ if and only if  
$\sigma$ is constant. This result may be used as a measure of quality of a silhouette polygon. 

A non-degenerate curve $\phi\subset M$ admits a parameterization satisfying equation \eqref{eq:AffineArcLengthDarboux}, unique up to a translation. 
The plane $\A=\A(t)$ generated by $\{\phi(t),\phi''(t)\}$ is called the {\it affine normal plane}, while the envelope $\B$ of these 
affine normal planes is a developable surface $\B$ called the {\it affine focal set} of the pair $\phi\subset M$ (\cite{Craizer15},\cite{Craizer16}). 
For silhouette curves relative to $O$, equation \eqref{eq:AffineArcLengthDarboux} reduces to equation \eqref{eq:CentroAffineArcLength}. 

We say that the polygon $\phi$ contained in the polyhedron $M$ is {\it equal-volume} if 
\begin{equation}\label{eq:DiscEqualVolumeDarboux}
\left[ \phi'(i-\tfrac{1}{2}),\phi'(i+\tfrac{1}{2}),\xi(i)\right]=1,
\end{equation}
for all $i$. Note that equation \eqref{eq:DiscEqualVolumeDarboux} is a discrete counterpart of equation \eqref{eq:AffineArcLengthDarboux}.
For such polygons, define the {\it affine normal plane} $\A(i)$ as the plane generated by $\{\phi(i),\phi''(i)\}$ and the {\it affine focal set} $\B=\B(\phi,M)$ as a discrete envelope
of these affine normal planes. For silhouette polygons $\phi$ relative to $O$, equation \eqref{eq:DiscEqualVolumeDarboux} reduces to
\begin{equation}\label{eq:EqualVolumeCA}
\left[ \phi(i-1)-O, \phi(i)-O,\phi(i+1)-O\right]=1,
\end{equation} 
which is a discrete counterpart of equation \eqref{eq:CentroAffineArcLength}.

The smooth curves $\phi\subset M$ whose affine focal set $\B$ reduces to a single line were 
characterized in \cite{Craizer16}. Consider a smooth planar curve $\Gamma(t)$ parameterized by affine arc-length and denote by $z(t)$ the {\it affine distance} or  {\it support function} 
of $\Gamma(t)$ with respect to some point $P\in\R^2$ (\cite{Cecil}). Then the affine focal set of the silhouette curve $\phi(t)=(\Gamma'(t),z(t))$ reduces to a single line and 
conversely, if $\B(\phi,M)$ is a single line, then $\phi$ is a silhouette curve obtained by this construction for some planar curve $\Gamma$ and some $P\in\R^2$. We prove in this paper a 
discrete counterpart of this characterization for equal-volume polygons contained in a polyhedron.

Consider a smooth curve $\Phi(t)$ in $3$-space parameterized by affine arc-length, i.e., satisfying equation \eqref{eq:AffineArcLength}. 
The planes through $\Phi$ parallel to $\{\Phi',\Phi'''\}$ are called {\it affine rectifying planes}
and the envelope of the affine rectifying planes $RS(\Phi)$ is called the {\it intrinsic affine binormal developable} (\cite{Izu2}). The characterization of curves $\Phi$
such that $RS(\Phi)$ is cylindrical is easily obtained from the characterization of curves $\phi=\Phi'$ whose affine focal set is a single line
(\cite{Craizer16},\cite{Izu2}).
A polygon $\Phi(i+\tfrac{1}{2})$ in $3$-space is said to be {\it equal-volume} if 
\begin{equation}\label{eq:EqualVolume}
\left[ \Phi'(i-1), \Phi'(i),\Phi'(i+1)\right]=1,
\end{equation}
for all $i$, which is equivalent to say that the difference polygon $\phi(i)=\Phi'(i)$ is equal-volume with respect to the origin. Although it is not clear
how to obtain a discrete version of the intrinsic affine binormal developable, we can obtain interesting consequences of 
the discrete characterization of polygons $\phi$ whose affine focal set is a single line.

We can also apply the equal-volume model in a projective setting.
Given a smooth planar curve $(\tilde\phi(t),1)$, there exists a projectively equivalent
curve $\phi(t)$ in $3$-space satisfying equation \eqref{eq:CentroAffineArcLength} with $O$ equal the origin. 
From this curve, we can define the projective length $pl(\tilde\phi)$ (see \cite{Guieu}).
For a planar polygon $(\tilde\phi(i),1)$, we can also obtain a projectively equivalent equal-volume polygon $\phi(i)$ in $3$-space and, from this polygon, 
we obtain two definitions for the projective length, $pl_1(\tilde\phi)$ and $pl_2(\tilde\phi)$, that
unfortunately do not coincide. Nevertheless, we prove that if the polygon is obtained from a dense enough sampling of a smooth curve, both the 
discrete projective length $pl_1(\tilde\phi)$ and $pl_2(\tilde\phi)$ are close to the projective length of the smooth curve.

The paper is organized as follows: In section 2 we review the smooth results for affine geometry of curves contained in surfaces, 
affine geometry of curves in $3$-space and projective geometry of planar curves. In section 3 we calculate affine invariants of equal-volume polygons contained in polyhedra. 
In section 4, we apply the results of section 3
to compute affine invariants for equal-volume polygons in $3$-space. In section 5 we discuss the projective length of a planar polygon.

\section{Affine geometry of smooth curves in $3$-space}

\subsection{Curves contained in surfaces}\label{sec:SmoothCM}

Let $\phi:I\to\R^3$ be a curve contained in a surface $M$ and $\xi$ a vector field tangent to $M$ and transversal to $\phi$. We shall assume that $\phi\subset M$
is non-degenerate, i.e., the osculating plane of $\phi$ does not coincide with the tangent plane of $M$ at any point. 
Under this hypothesis, there exists a vector field $\xi(t)$, unique up to scalar (non-constant) multiple, such that $\xi'(t)$ is tangent to $M$, 
for any $t\in I$. The vector field $\xi$ determines a unique direction tangent to $M$, which is called the {\it Darboux direction} along $\phi$. 
In the Darboux direction, there exists a vector field $\xi(t)$, unique up to a constant multiple, such that $\xi'(t)$ is tangent to $\phi(t)$, for any $t\in I$. We call this vector field the 
{\it parallel Darboux vector field}. The parallel Darboux vector field satisfies equation \eqref{eq:DefineSigma}, for some scalar function $\sigma$. 

The envelope of tangent planes is the developable surface 
$$
x(t,u)=\phi(t)+u\xi(t).
$$
This surface is called the {\it Osculating Tangent Developable Surface of $M$ along $\phi$} and will be denoted $\E$ (\cite{Craizer15},\cite{Izu3}). The 
surface $\E$ is a cone if and only if $\sigma$ is constant. In this case, the vertex of the cone is given by $O=\phi+\sigma^{-1}\xi$ and the curve $\phi$
is a silhouette curve from the point of view of $O$.

Under the non-degeneracy hypothesis, there exists a parameterization $\phi(t)$ of $\phi$, unique up to a translation, such that equation \eqref{eq:AffineArcLengthDarboux} holds. 
The plane $\A(t)$ generated by $\{\xi(t),\phi''(t)\}$ is called the {\it affine normal plane} of $\phi\subset M$. 
Condition \eqref{eq:AffineArcLengthDarboux} is equivalent to $\phi'''(t)$ tangent to $M$. Thus we can write 
\begin{equation}\label{eq:Frenet1}
\phi'''(t)=-\rho(t)\phi'(t)+\tau(t)\xi(t), 
\end{equation} 
for some scalar functions $\rho$ and $\tau$. 

There exists a basis $\{\xi(t),\eta(t)\}$ of the affine normal plane $\mathcal{A}(t)$ with $\eta$ parallel, i.e., $\eta'$ tangent to $\phi$. In fact,
define the vector field $\eta$ by $\eta=\phi''+\lambda\xi$, where $\lambda'=-\tau$. Taking $\mu=\rho+\lambda\sigma$, we obtain the equation
\begin{equation}\label{eq:EtaParallel}
\eta'(t)=-\mu(t)\phi'(t),
\end{equation} 
which in particular says that $\eta$ is parallel. 
The {\it affine focal set} $\B$, or {\it affine evolute}, is the envelope of affine normal planes. It is the developable surface generated by the lines 
passing through $O=\phi+\sigma^{-1}\xi$ and $Q=\phi+\mu^{-1}\eta$. The affine focal set reduces to a single line
if and only if $\sigma$ and $\mu$ are constant (\cite{Craizer16}).

If $\phi$ is contained in a plane $L$, then $\tau(t)=0$ for any $t\in I$ and conversely, if $\tau(t)=0$ for any $t\in I$, then $\phi$ is planar. Denote by $n$ a euclidean unitary normal to $L$
and let $\xi$ be the vector field in the Darboux direction such that $\xi\cdot n=1$, where $\cdot$ denotes the usual inner product.  Then $\xi$ is a parallel Darboux vector field.
In this case, the adapted parameter $t$ corresponds to the affine arc-length parameter and $\rho(t)$ is the affine curvature of $\phi\subset L$. 
For planar curves, $\lambda=0$ and so $\phi''=\eta\subset L$ is parallel. Then the set $\mathcal{B}\cap L$ coincides with the affine evolute of the planar curve $\phi\subset L$ (\cite{Izu1}). 

For silhouette curves, $\xi(t)=\phi(t)-O$ and so equation \eqref{eq:AffineArcLengthDarboux} becomes \eqref{eq:CentroAffineArcLength}, i.e., $\phi(t)$ is parameterized by {\it centro-affine arc-length}.
Assuming $O$ equals the origin, equation \eqref{eq:Frenet1} becomes
\begin{equation}\label{eq:FrenetCA}
\phi'''(t)=-\rho(t)\phi'(t)+\tau(t)\phi(t).
\end{equation}
Moreover $\mu=\rho-\lambda$, which implies $\mu'=\rho'+\tau$.

\paragraph{Curves whose affine focal set $\B$ reduces to a single line} The affine focal set reduces to a single line if and only if $\mu$ and $\sigma$ are constant. 
Since $\sigma$ is constant, $\phi$ is necessarily a silhouette curve from the point of view of $O$, that we shall assume to be the origin.
The condition $\mu$ constant can be written as $\rho'+\tau=0$. In this case equation \eqref{eq:FrenetCA}
becomes 
$\phi'''=- \left(  \rho\phi\right)'$, which is equivalent to
$$
\phi''(t)=-\rho(t)\phi(t)+Q,
$$
for some constant vector $Q$. Assuming that $Q=(0,0,1)$ and writing $\phi(t)=\left( \gamma(t),z(t) \right)$, this equation becomes
\begin{equation}\label{eq:MuConstant}
\gamma''(t)=-\rho(t)\gamma(t);\ \ z''(t)=-\rho(t)z(t)+ 1.
\end{equation}

Consider a convex planar curve $\Gamma(t)$. Assume that $\Gamma(t)$ is parameterized by affine arc-length, i.e., $\left[\Gamma'(t),\Gamma''(t)  \right]=1$, and let 
$\rho(t)$ denotes the affine curvature of $\Gamma$, i.e., $\Gamma'''(t)=-\rho(t)\Gamma'(t)$, $t\in I$. Let $\gamma(t)=\Gamma'(t)$ and denote by $z(t)=\left[\Gamma(t)-P,\gamma(t)\right]$ the
{\it affine distance}, or {\it support function}, of $\Gamma$ with respect to a point $P\in\R^2$ (\cite{Cecil}).

The following proposition was proved in \cite{Craizer16}:
\begin{Proposition}\label{prop:MuConstant}
The affine focal set $\B$ of the curve $\phi(t)=\left( \gamma(t), z(t)  \right)$ is a single line, and conversely, any curve $\phi$ whose affine focal set
is a single line can be obtained as above, for some planar curve $\Gamma$ and some point $P\in\R^2$. 
\end{Proposition}

\subsection{Curves in $3$-space}\label{sec:SmoothSpatial}

Consider now a curve $\Phi$ in the $3$-space, without being contained in a given surface $M$. 
We say that a parameterization $\Phi(t)$ of $\Phi$ is by {\it affine arc-length} if formula \eqref{eq:AffineArcLength} holds. This condition
implies that $\Phi''''(t)$ belongs to the plane generates by $\phi''(t)$ and $\phi'(t)$ and thus we obtain equation 
\begin{equation}\label{eq:FrenetSp}
\Phi''''(t)=-\rho(t)\Phi''(t)+\tau(t)\Phi'(t),
\end{equation}
for some scalar functions $\rho$ and $\tau$. Writing $\phi(t)=\Phi'(t)$, we observe that equation \eqref{eq:FrenetSp} reduces to
\eqref{eq:FrenetCA}.

The plane passing through $\Phi(t)$ and generated by $\{\Phi'(t),\Phi'''(t)\}$ is called 
{\it affine rectifying plane} and the envelope $RS(\Phi)$ of the affine rectifying planes is called the {\it intrinsic affine binormal developable} of $\Phi$. 
It is proved in \cite{Izu2} that $RS(\Phi)$ is cylindrical if and only if $\rho'+\tau=0$.

\paragraph{Curves with $\mu$ constant} The condition $\mu$ constant is equivalent to $\rho'+\tau=0$. Consider a convex planar
curve $\Gamma$ parameterized by affine arc-length and let
\begin{equation*}
Z(t)=\int_{t_0}^t \left[\Gamma(s)-P,\Gamma'(s)\right] ds.
\end{equation*}
Then $Z(t)$ represents the area of the planar region bounded by $\Gamma(s)$, $t_0\leq s\leq t$, and the segments $P\Gamma(t_0)$ and $P\Gamma(t)$. The following proposition
is a direct consequence of proposition \ref{prop:MuConstant}:

\begin{Proposition}
For the  curve $\Phi(t)=(\Gamma(t), Z(t))$, $\mu$ is constant, and conversely, any curve $\Phi$ in $3$-space with $\mu$ constant
is obtained by this construction, for some convex planar curve $\Gamma$ and some point $P\in\R^2$. 
\end{Proposition}

\subsection{Projective invariants}\label{sec:SmoothProjective}

Consider a parameterized planar curve $\tilde\phi(t)$, $t\in I$, without inflection points. 
Any curve of the form $\phi(t)=a(t)\tilde\phi(t)$ is projectively equivalent to $\tilde\phi(t)$ and is called a {\it representative} of $\tilde\phi(t)$. 
It turns out that there exists a representative $\phi(t)$ of $\tilde\phi(t)$ satisfying formula \eqref{eq:CentroAffineArcLength} with $O$ equal to the origin.
Then $\phi'''(t)$ belongs to the space generated by $\{\phi'(t),\phi(t)\}$ and equation \eqref{eq:FrenetCA} holds, 
for some scalar functions $\rho$ and $\tau$. 

The quantity $\rho'(t)+2\tau(t)$ is projectively invariant and $\rho'(t)+2\tau(t)=0$ if and only if $\phi$ is contained in a quadratic cone. In fact, 
\begin{equation}\label{eq:ProjectiveLength}
pl(\tilde\phi)=\int_{I} (\rho'(t)+2\tau(t))^{1/3}dt
\end{equation}
is the projective length of $\tilde\phi$ (see \cite{Guieu}).

\section{Affine geometry of equal-volume polygons contained in polyhedra}

In this section, we obtain discrete counterparts of the results of section \ref{sec:SmoothCM}. The derivatives are replaced by differences, and so 
for a function $f:\{1,...,N\}\to\R^k$, we denote 
\begin{equation*}
f'(i+\tfrac{1}{2})= f(i+1)-f(i),\ \ f''(i)=f'(i+\tfrac{1}{2})-f'(i-\tfrac{1}{2}),
\end{equation*}
and so on. 

\subsection{Basic model}

Consider a polyhedron $M$ whose faces are planar quadrilaterals and let $\phi$ be a polygonal line such that each of its sides are connecting opposite edges
of a face of $M$. We shall denote by $\phi(i)$, $1\leq i\leq N$, the vertices of such polygon and by  
$\xi(i)$ a vector in the direction of the edges of $M$ containing $\phi(i)$. Edges of $M$ that don't intersect $\phi$ are not important in our model (see Figure \ref{fig:Polygonal}).

\begin{figure}[htb]
 \centering
 \includegraphics[width=0.60\linewidth]{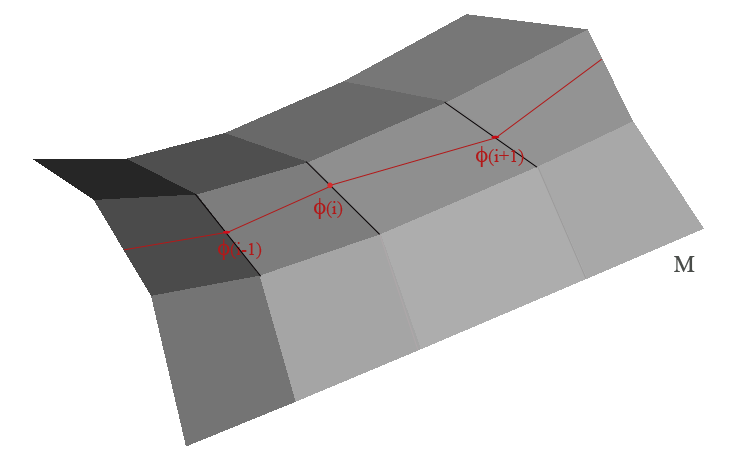}
 \caption{Polygonal line $\phi$ contained in polyhedron $M$.}
\label{fig:Polygonal}
\end{figure}

We shall denote by $T_{i+1/2}M$ the face of $M$ that contains the side $i+1/2$. 
By the planar quadrilaterals hypothesis, the vectors $\phi'(i+\tfrac{1}{2})$, $\xi(i)$ and $\xi(i+1)$ belong to $T_{i+1/2}M$, 
which is a discrete counterpart of the Darboux condition $\xi'$ tangent to $M$. It is clear that there exists $\xi$,  unique up to a multiplicative constant,
such that $\xi'(i+1/2)$ is parallel to $\phi'(i+1/2)$. This vector field is the {\it parallel Darboux} vector field of $\phi\subset M$ and equation \eqref{eq:DiscSigma} holds, for some scalar function $\sigma$  (see Figure \ref{fig:XiVector}).

\begin{figure}[htb]
 \centering
 \includegraphics[width=0.60\linewidth]{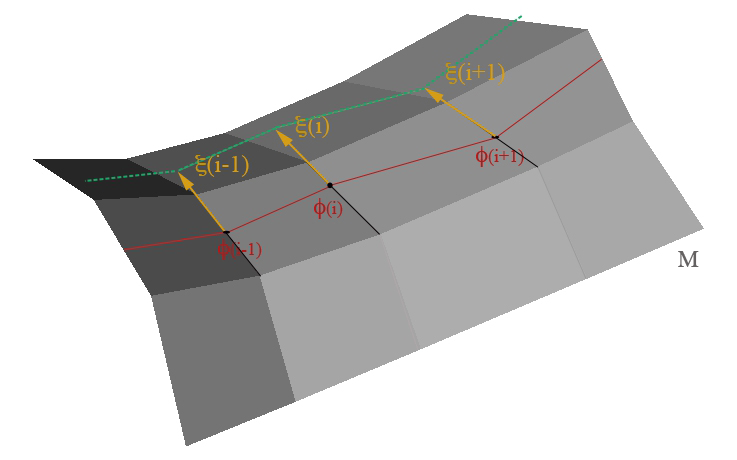}
 \caption{Parallel Darboux vectors field $\xi$. The segments connecting the endpoints of the vectors $\xi(i)$ are parallel to the sides of $\phi$.}
\label{fig:XiVector}
\end{figure}

\subsection{Osculating developable polyhedron}

The line $x(i,u)=\phi(i)+u\xi(i)$, $u\in\R$,
is the support line of the edge of the polyhedron $M$ that contains $\phi$. Thus $x(i,u)$ and $x(i+1,u)$ are co-planar
and denote by $O(i+\tfrac{1}{2})$ the intersection point of these lines. We have that
\begin{equation}\label{eq:DefineO}
O(i+\tfrac{1}{2})=\phi(i)+\sigma^{-1}(i+\tfrac{1}{2})\xi(i)= \phi(i+1)+\sigma^{-1}(i+\tfrac{1}{2})\xi(i+1).
\end{equation}
The {\it osculating developable polyhedron} $\E$ is the polyhedron whose face $i+1/2$ is the region of $T_{i+1/2}M$ bounded 
by $x(i,u)$ and $x(i+1,u)$ and containing the side $i+1/2$ of $\phi$ (see Figure \ref{fig:ODP}).

\begin{figure}[htb]
 \centering
 \includegraphics[width=0.60\linewidth]{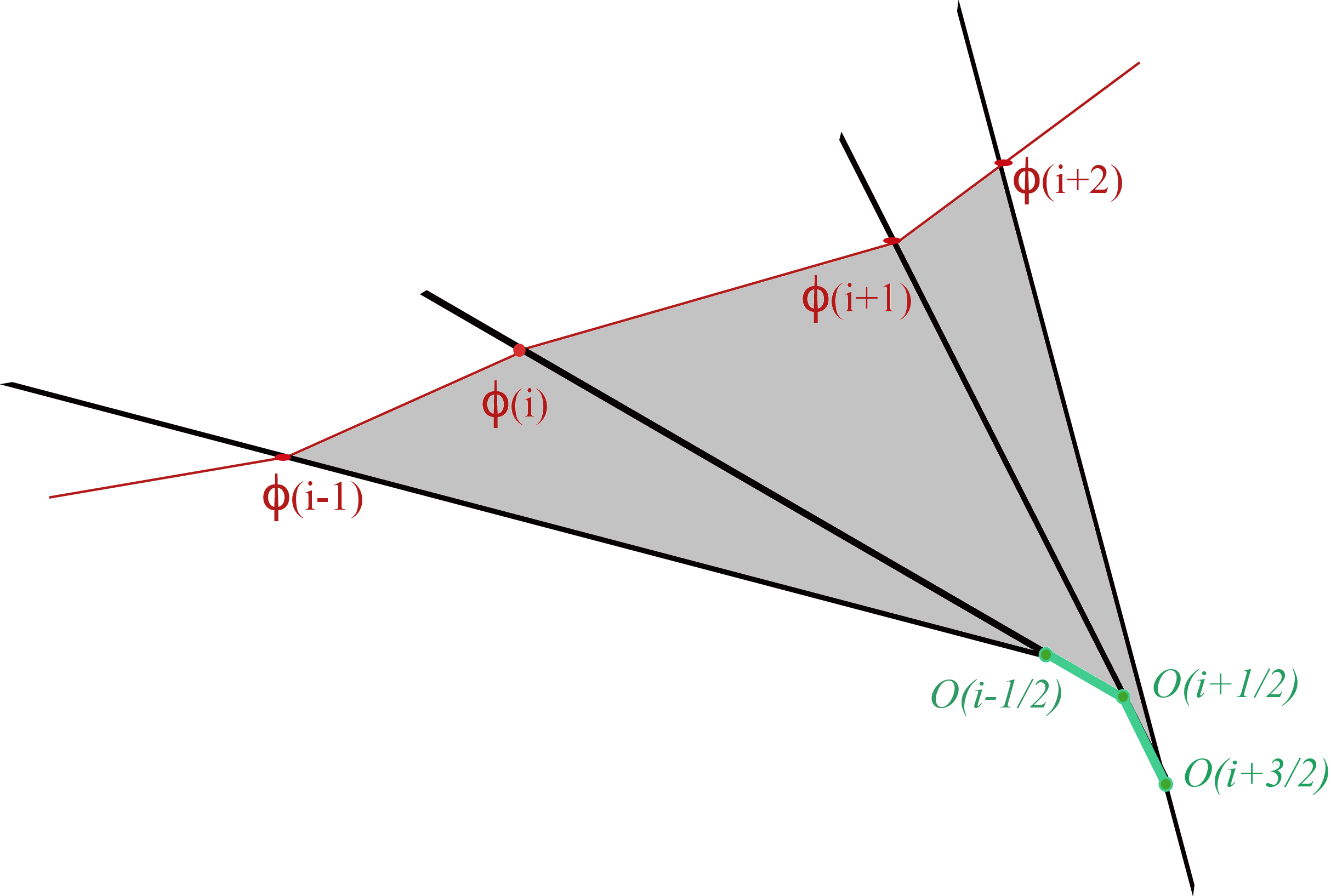}
 \caption{Osculating developable polyhedron $\E$}
\label{fig:ODP}
\end{figure}

We have the following proposition:

\begin{Proposition}\label{prop:DiscreteODS}
The following statements are equivalent:
\begin{enumerate}
\item $\sigma(i+\tfrac{1}{2})$ does not depend on $i$.
\item The point $O(i+\tfrac{1}{2})$ does not depend on $i$. 
\item $\E$ is a cone.
\end{enumerate}
\end{Proposition}

\begin{proof}
Observe that
\begin{equation}\label{eq:DerivO}
O(i+\tfrac{1}{2})-O(i-\tfrac{1}{2})=\left( \sigma^{-1}(i+\tfrac{1}{2})-\sigma^{-1}(i+\tfrac{1}{2}) \right) \xi(i)
\end{equation}
Thus (1) is equivalent to (2). The equivalence between (2) and (3) is obvious.
\end{proof}

The polygons $\phi\subset M$ for which $\E$ reduces to a point $O$ is the object of study of the centro-affine geometry.
In this case, the polygon $\phi$ can be thought as a silhouette polygon of $M$ from the point of view of $O$.

\subsection{Equal-volume polygons}

We say that the polygon $\phi$ contained in the polyhedron $M$ is {\it equal-volume} if equation \eqref{eq:DiscEqualVolumeDarboux} holds. 

\begin{lemma}\label{lemma:Phi3LinhasTangente}
The polygon $\phi\subset M$ is equal-volume if and only if 
\begin{equation*}\label{eq:DiscAdapted}
\phi'''(i+\tfrac{1}{2})\in T_{i+1/2}M.
\end{equation*}
\end{lemma}

\begin{proof}
Equation \eqref{eq:DiscEqualVolumeDarboux} is equivalent to
$$
\left[\phi'(i+\tfrac{1}{2}),\phi'(i+\tfrac{3}{2}),\xi(i+1)  \right]-\left[\phi'(i-\tfrac{1}{2}),\phi'(i+\tfrac{1}{2}),\xi(i)  \right]=0,
$$
for each $i$, which is equivalent to
$$
\left[ \phi'(i+\tfrac{1}{2}), \phi'(i+\tfrac{3}{2}), \xi(i+1)-\xi(i)\right]-\left[\phi'(i+\tfrac{1}{2}),\phi''(i)-\phi''(i+1),\xi(i)  \right]=0.
$$
By the parallel Darboux condition, the first parcel is zero and thus the above condition is equivalent to
$$
\left[\phi'(i+\tfrac{1}{2}),\phi''(i)-\phi''(i+1),\xi(i)  \right]=0,
$$
which is clearly equivalent to $\phi'''(i+1/2)$ belongs to $T_{i+1/2}M$. 
\end{proof}

We shall assume along the paper that the polygon $\phi\subset M$ is equal-volume. By the above lemma we can write
\begin{equation}\label{eq:DiscFrenet1}
\left\{
\begin{array}{c}
\phi'''(i+\tfrac{1}{2})=-\rho_{2}(i) \phi'(i+\tfrac{1}{2})  +\tau(i+\tfrac{1}{2})\xi(i+1)\\
\phi'''(i+\tfrac{1}{2})=-\rho_{1}(i+1) \phi'(i+\tfrac{1}{2})  +\tau(i+\tfrac{1}{2})\xi(i)
\end{array}
\right.
\end{equation}
for some scalar functions $\rho_1$, $\rho_2$ and $\tau$ satisfying the 
compatibility equation
\begin{equation}\label{eq:Compatibility}
-\tau(i+\tfrac{1}{2})\sigma(i+\tfrac{1}{2})=\rho_{2}(i)-\rho_{1}({i+1}).
\end{equation}
Equations \eqref{eq:DiscFrenet1} are discrete counterparts of equation \eqref{eq:Frenet1}.

\bigskip\noindent
\begin{remark} Starting from a general polygon $\phi\subset M$, we may obtain an equal-volume polygon $\bar\phi\subset\bar {M}$ by the following inductive algorithm (see Figure \ref{fig:EqualVolume}): 

\begin{enumerate}
\item Let $(\bar\phi(i),\bar\xi(i))=(\phi(i),\xi(i))$, for $i=1,2,3$.
\item Given the pair $(\bar\phi,\bar\xi)$ at $i-1$, $i$ and $i+1$, consider a plane parallel to $T_{i+1/2}\bar{M}$ through $\bar\phi(i-1)$ and let $\bar\phi(i+2)$ be the intersection of this plane with the polygonal line $\phi$.
\item The direction of the vector $\bar\xi(i+2)$ is obtained by linear interpolation of $\xi(k)$ and $\xi(k+1)$, where $k+\tfrac{1}{2}$ is the index of the side of $\phi$ containing $\bar\phi(i+2)$. Thus $\bar\xi(i+2)\in T_{k+1/2}M$.

\end{enumerate}

\begin{figure}[htb]
 \centering
 \includegraphics[width=0.50\linewidth]{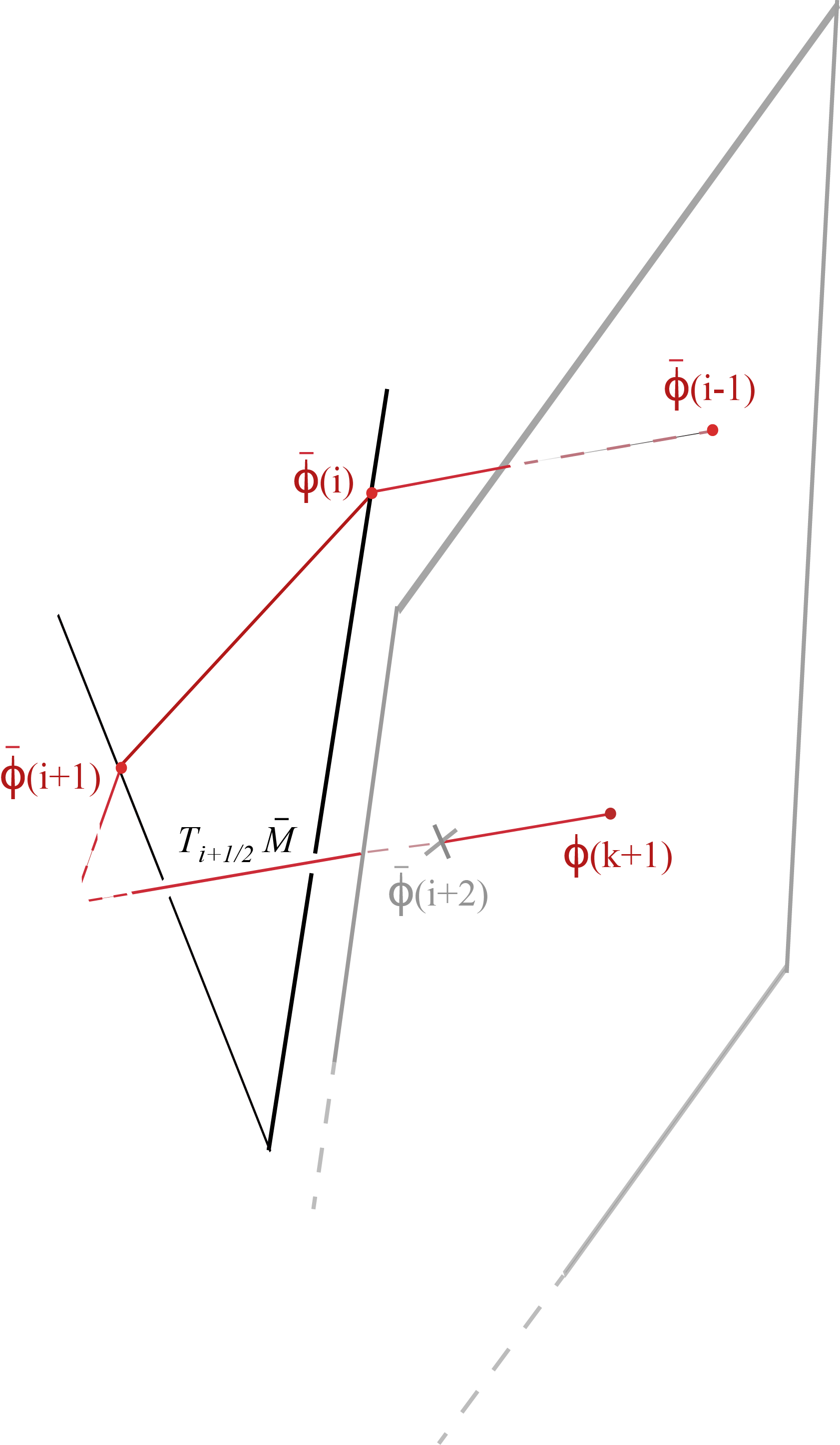}
 \caption{Algorithm to construct an equal-volume polygon}
\label{fig:EqualVolume}
\end{figure}

\end{remark}

\subsection{Discrete affine focal set}

Take any $\lambda$ satisfying $\lambda(i)-\lambda(i+1)=\tau(i+\tfrac{1}{2})$ and define
\begin{equation}\label{eq:DefineEta}
\eta(i)=\phi''(i)+\lambda(i)\xi(i).
\end{equation}
Define also
\begin{equation}\label{eq:DefineMu}
\mu(i+\tfrac{1}{2})= \rho_1(i+1)+\sigma(i+\tfrac{1}{2})\lambda(i+1)=\rho_2(i)+\sigma(i+\tfrac{1}{2})\lambda(i).
\end{equation}

\medskip\noindent

\begin{lemma}
The following discrete counterpart of equation \eqref{eq:EtaParallel} holds:
\begin{equation}\label{eq:DiscEtaParallel}
\eta'(i+\tfrac{1}{2})=-\mu(i+\tfrac{1}{2})  \phi'(i+\tfrac{1}{2}).
\end{equation}
In particular, $\eta$ is parallel.
\end{lemma}
\begin{proof}
We have that
$$
\eta'(i+\tfrac{1}{2})=\phi'''(i+\tfrac{1}{2})+\lambda'(i+\tfrac{1}{2})\xi(i+1)+\lambda(i)\xi'(i+\tfrac{1}{2})
$$
$$
=-\rho_2(i)\phi'(i+\tfrac{1}{2})+\lambda(i)\xi'(i+\tfrac{1}{2})=-\left(\rho_2(i)+\sigma(i+\tfrac{1}{2})\lambda(i)\right)\phi'(i+\tfrac{1}{2}),
$$
thus proving the lemma.
\end{proof}

Define
\begin{equation}\label{eq:DefineQ}
Q(i+\tfrac{1}{2})=\phi(i)+\mu^{-1}(i+\tfrac{1}{2})\eta(i)=\phi(i+1)+\mu^{-1}(i+\tfrac{1}{2})\eta(i+1),
\end{equation}
and denote by $l(i+\tfrac{1}{2})$ the line connecting $O(i+\tfrac{1}{2})$ and $Q(i+\tfrac{1}{2})$, where 
$O(i+\tfrac{1}{2})$ is defined by equation \eqref{eq:DefineO}. 

The {\it discrete affine focal set} $\B$ is the polyhedron with edges $l(i+\tfrac{1}{2})$, $i=1,...,N-1$, and faces contained in $\A(i)$, $i=2,...,N-1$, bounded 
by $l(i-\tfrac{1}{2})$ and $l(i+\tfrac{1}{2})$ containing the segments $Q(i-\tfrac{1}{2})Q(i+\tfrac{1}{2})$ and $O(i-\tfrac{1}{2})O(i+\tfrac{1}{2})$ (see Figure \ref{fig:DiscreteFocalSet}).

\begin{figure}[htb]
 \centering
 \includegraphics[width=0.80\linewidth]{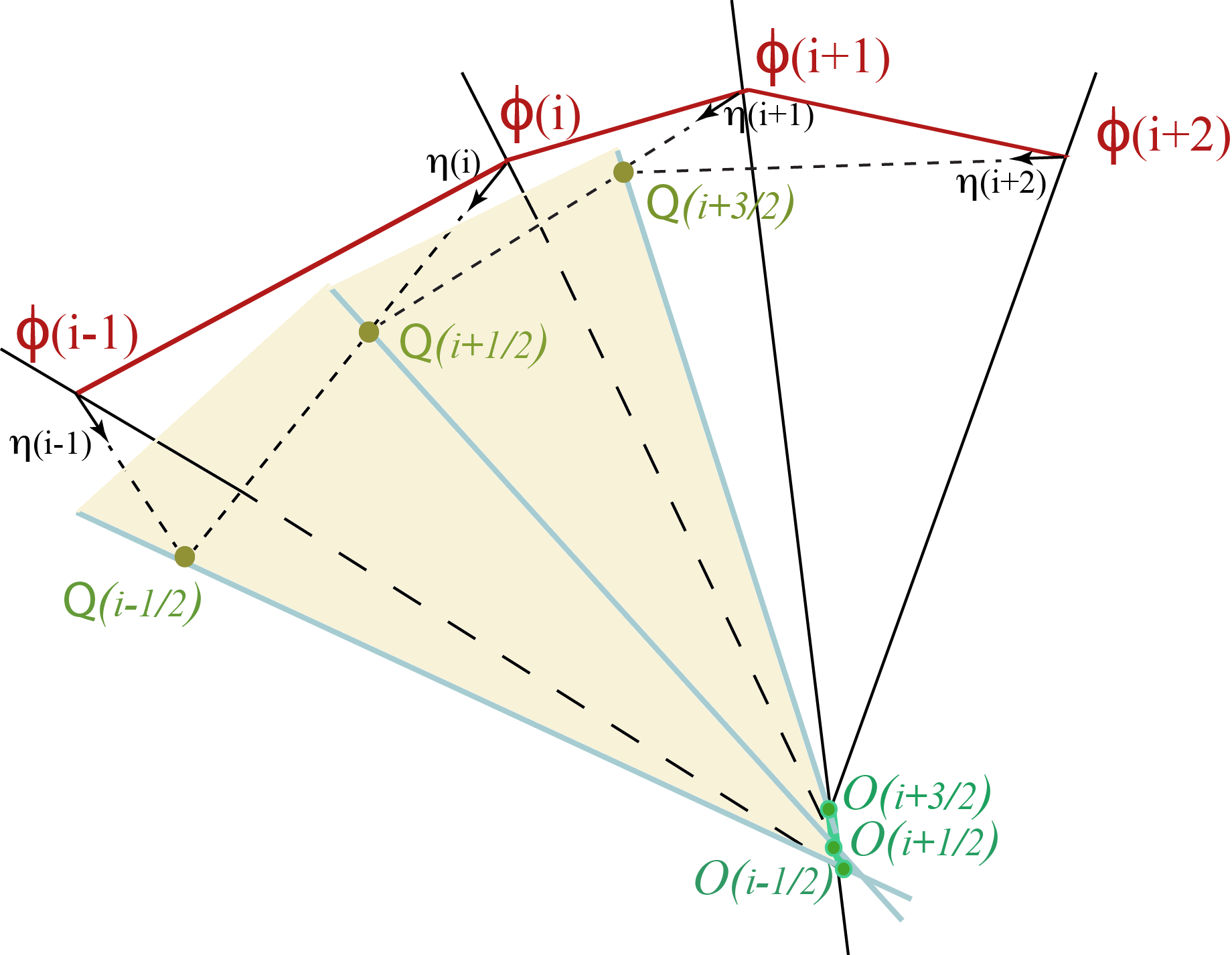}
 \caption{Discrete affine focal set}
\label{fig:DiscreteFocalSet}
\end{figure}

\begin{Proposition}\label{prop:DiscSingleLine}
The following statements are equivalent:
\begin{enumerate}
\item $\sigma$ and $\mu$ are constant.
\item The points $O$ and $Q$ are fixed. 
\item The discrete affine focal set $\B$ reduces to a single line. 
\end{enumerate}
\end{Proposition}

\begin{proof}
By proposition \ref{prop:DiscreteODS}, $\sigma$ constant is equivalent to $O$ fixed. From equation \eqref{eq:DefineQ} we obtain
\begin{equation}\label{eq:DerivQ}
Q(i+\tfrac{1}{2})-Q(i-\tfrac{1}{2})=\left(  \mu^{-1}(i+\tfrac{1}{2})-  \mu^{-1}(i-\tfrac{1}{2})   \right) \eta(i),
\end{equation}
which implies that $\mu$ is constant if and only if $Q$ is fixed. Thus (1) and (2) are equivalent. It is obvious that (2) implies (3) and so it remains to prove that (3) implies (2).
If $O$ and $Q$ were not both fixed, then equations \eqref{eq:DerivO} and \eqref{eq:DerivQ} say that $O$ or $Q$ are not changing in the direction of $Q-O$. Thus
$\B$ would not be a single line.
\end{proof}

\subsection{Planar polygons} 

\begin{lemma}
A polygon $\phi$ is contained in a plane $L$ if and only if $\tau=0$.
\end{lemma}
\begin{proof}
Observe that $\tau(i+\tfrac{1}{2})=0$ if and only if the points $\phi(i-1)$, $\phi(i)$, $\phi(i+1)$ and $\phi(i+2)$ are co-planar. 
\end{proof}

Denote by $n$ a euclidean unitary normal to $L$ and let $\xi$ be the vector field in the Darboux direction such that $\xi\cdot n=1$, where $\cdot$ denotes the usual inner product.  
Then $\xi$ is a parallel Darboux vector field. In this case equation \eqref{eq:DiscEqualVolumeDarboux} can be written as 
\begin{equation*}\label{eq:EqualArea}
\left[\phi'(i-\tfrac{1}{2}),\phi'(i+\tfrac{1}{2}) \right]=1,
\end{equation*}
where $[\cdot,\cdot]$ denotes determinant in the plane $L$. Thus $\phi\subset L$ is an equal-area polygon and $\rho$ is its discrete affine curvature (\cite{Craizer13},\cite{Kaferbock14}). 
The set $\B\cap L$ is exactly the discrete affine evolute of the planar equal-area polygon $\phi\subset L$ (\cite{Craizer13}) (see Figure \ref{fig:Planar}).

\begin{figure}[htb]
 \centering
 \includegraphics[width=0.80\linewidth]{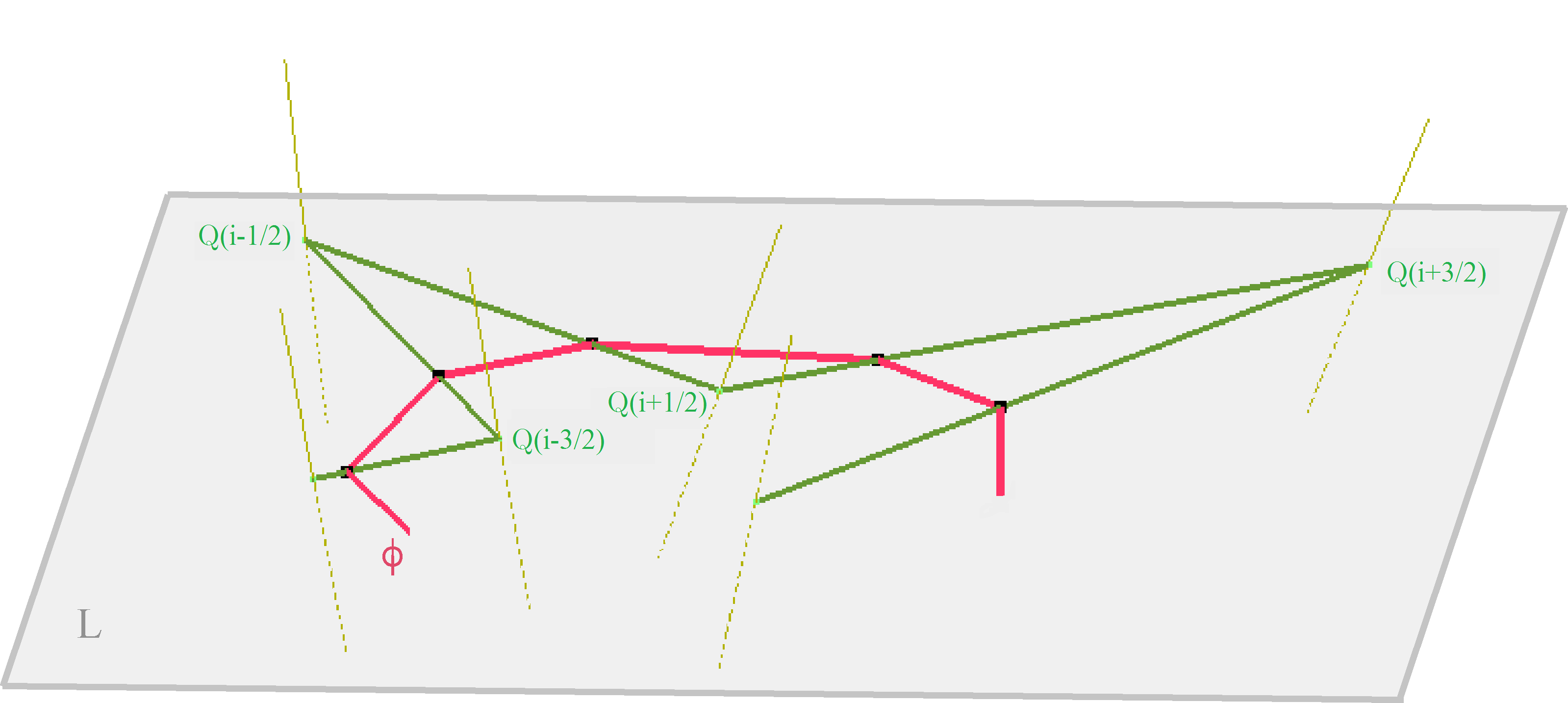}
 \caption{For a planar curve $\phi$, the set $\B\cap L$ coincides with the discrete affine evolute of $\phi$.}
\label{fig:Planar}
\end{figure}

\subsection{Silhouette polygons}

Assume that  $\phi\subset\R^3$ is a silhouette polygon from the point of view of $O$, that we assume to be the origin. In this case, equation \eqref{eq:DiscEqualVolumeDarboux}
becomes equation \eqref{eq:EqualVolumeCA}. By Lemma \ref{lemma:Phi3LinhasTangente}, this condition is equivalent
to $\phi'''(i+\tfrac{1}{2})$ belongs to the plane generated by $\phi(i)$ and $\phi(i+1)$.

Since $\xi(i)=\phi(i)$, we have that $\sigma(i+\tfrac{1}{2})=-1$, $i=1,...,N$. The Frenet equations \eqref{eq:DiscFrenet1} reduce to
\begin{equation}\label{eq:DiscFrenetCA1}
\left\{
\begin{array}{c}
\phi'''(i+\tfrac{1}{2})=-\rho_2(i)\phi'(i+\tfrac{1}{2})+\tau(i+\tfrac{1}{2})\phi(i+1)\\
\phi'''(i+\tfrac{1}{2})=-\rho_1(i+1)\phi'(i+\tfrac{1}{2})+\tau(i+\tfrac{1}{2})\phi(i),
\end{array}
\right.
\end{equation}
while equation \eqref{eq:DefineMu} becomes
\begin{equation}\label{eq:DiscMu}
\mu(i+\tfrac{1}{2})=\rho_1(i+1)-\lambda(i+1)=\rho_2(i)-\lambda(i).
\end{equation}
We have also that
\begin{equation}\label{eq:DiscMuLinha}
\mu'(i)=\rho_1'(i+\tfrac{1}{2})+\tau(i+\tfrac{1}{2})= \rho_2'(i-\tfrac{1}{2})+\tau(i-\tfrac{1}{2}).
\end{equation}

\subsection{Polygons whose discrete affine focal set reduces to a line}

By proposition \ref{prop:DiscSingleLine}, $\B$ reduces to a single line if and only if $\mu$ and $\sigma$ are constant. Since $\sigma$ 
is constant, $\phi$ is a silhouette polygon. By formula \eqref{eq:DiscMuLinha}, the condition $\mu$ constant is equivalent 
to $\rho_1'+\tau=\rho_2'+\tau=0$.

Assume $\mu(i+\tfrac{1}{2})=\mu_0$ constant. Then equation \eqref{eq:DiscEtaParallel} implies that
$$
\eta(i)=-\mu_0 \phi(i)+Q,
$$
for some constant vector $Q$. Assume $Q=(0,0,1)$ and write $\phi(i)=\left(  \gamma(i), z(i)  \right)$. Then, using equation \eqref{eq:DefineEta} we obtain 
$$
\gamma''(i)+\lambda(i) \gamma(i)=-\mu_0 \gamma(i), \ \ z''(i)+\lambda(i)z(i)=-\mu_0z(i)+1,
$$
and so
\begin{equation}\label{eq:Difference}
\gamma''(i)=-(\lambda(i)+\mu_0)\gamma(i),\ \ z''(i)=-(\lambda(i)+\mu_0)z(i)+1.
\end{equation}
Observe that
$$
\left[ \gamma(i),\gamma(i+1)\right] -\left[ \gamma(i-1),\gamma(i)\right] =\left[\gamma(i),\gamma''(i) \right]=0, 
$$
and so $[\gamma(i),\gamma(i+1)]=c$, for some constant $c$. By rescaling $\phi$ we may assume that $c=1$.

Denote by $\Gamma(i+\tfrac{1}{2})$ a polygon such that $\Gamma'(i)=\gamma(i)$. 
Then 
$$
\Gamma'''(i)=-(\lambda(i)+\mu_0)\Gamma'(i),
$$ 
and so $\Gamma$ is an equal-area polygon with discrete affine curvature $\lambda(i)+\mu_0$. The {\it affine distance} or {\it support function} of $\Gamma$ with respect to a point $P\in\R^2$ is given by 
\begin{equation}\label{eq:Discz}
z(i)=\left[\Gamma(i+\tfrac{1}{2})-P,\gamma(i)\right]=\left[\Gamma(i-\tfrac{1}{2})-P,\gamma(i)\right]
\end{equation}
(see Figure \ref{fig:Support}, left).

\begin{figure}[htb]
\centering
\subfigure{
\includegraphics[width=.45\textwidth]{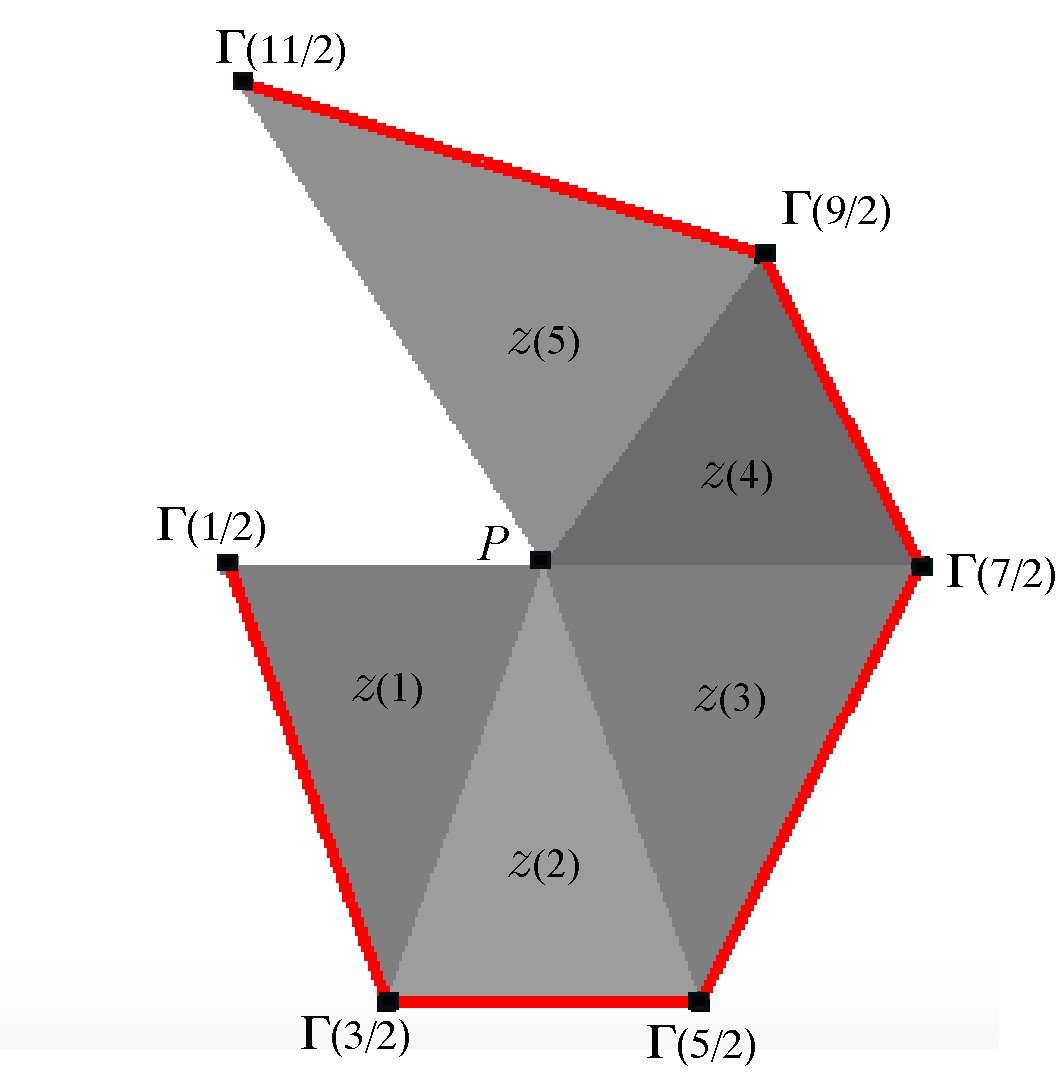}  }
\subfigure{
\includegraphics[width=.45\textwidth]{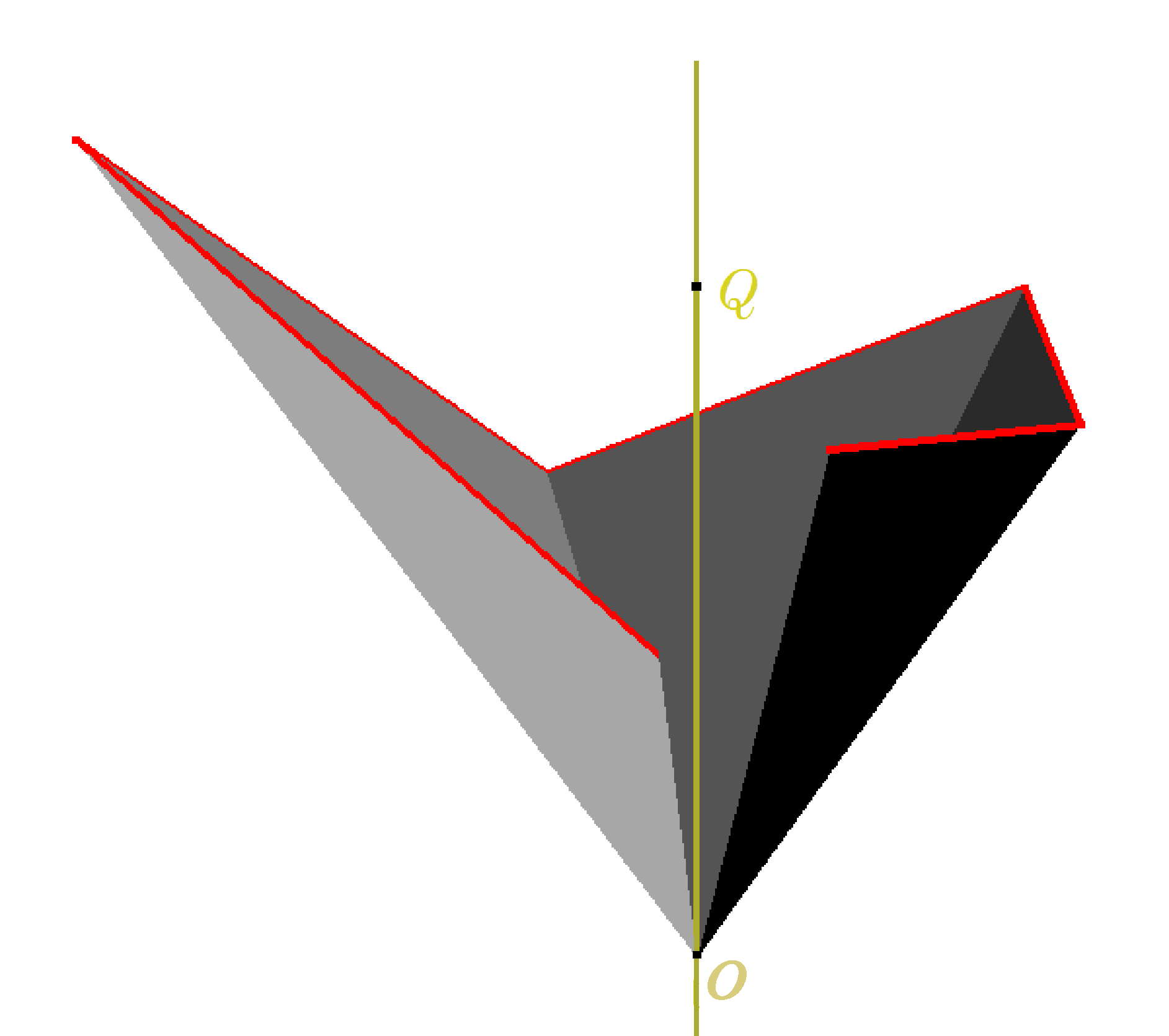} }
\caption{A planar equal-area polygon $\Gamma$ and its support function $z$ (left). The corresponding spatial curve $\phi$ and its affine focal set (right).}
\label{fig:Support}
\end{figure}

\begin{Proposition}\label{prop:Main}
The polygonal line $\phi(i)=(\gamma(i),z(i))$ (see Figure \ref{fig:Support}, right) satisfies equation \eqref{eq:Difference}, and conversely, any solution 
of the difference equation \eqref{eq:Difference} is obtained by this construction, for some planar polygon $\Gamma(i+\tfrac{1}{2})$ and some point $P\in\R^2$.
\end{Proposition}
\begin{proof}
Observe first that
$$
z'(i+\tfrac{1}{2})= \left[ \Gamma(i+\tfrac{1}{2})-P, \gamma(i+1) \right]-\left[ \Gamma(i+\tfrac{1}{2})-P, \gamma(i) \right]
$$
$$
=\left[  \Gamma(i+\tfrac{1}{2})-P, \gamma'(i+\tfrac{1}{2} )   \right].
$$
Thus 
$$
z''(i)=\left[  \Gamma(i+\tfrac{1}{2})-P, \gamma'(i+\tfrac{1}{2} )   \right]-\left[  \Gamma(i-\tfrac{1}{2})-P, \gamma'(i-\tfrac{1}{2} )   \right]
$$
$$
=\left[ \gamma(i), \gamma'(i+\tfrac{1}{2}) \right]+\left[  \Gamma(i-\tfrac{1}{2})-P, \gamma''(i)   \right]
$$
$$
=1-(\lambda(i)+\mu_0)z(i),
$$
thus proving that $(\phi(i),z(i))$ satisfies equation \eqref{eq:Difference}. Since 
$P$ has two degrees of freedom, this is the general solution of the second order difference equation \eqref{eq:Difference}.
\end{proof}

\section{Polygons in $3$-space}

In this section, we obtain discrete counterparts of the results of section \ref{sec:SmoothSpatial}. Consider a polygon $\Phi(i+\tfrac{1}{2})$ in $3$-space, 
without being contained in any polyhedron $M$. The polygon $\Phi$ is equal-volume, i.e., satisfies equation \eqref{eq:EqualVolume}, if and only if 
the difference polygon $\phi(i)=\Phi'(i)$ is equal-volume with respect to the origin.

\subsection{Frenet equations}

For equal-volume polygons $\Phi$, Frenet equations \eqref{eq:DiscFrenetCA1} are written as 
\begin{equation}\label{eq:DiscFrenetSP1}
\left\{
\begin{array}{c}
\Phi''''(i+\tfrac{1}{2})=-\rho_2(i)\Phi''(i+\tfrac{1}{2})+\tau(i+\tfrac{1}{2})\Phi'(i+1)\\
\Phi''''(i+\tfrac{1}{2})=-\rho_1(i+1)\Phi''(i+\tfrac{1}{2})+\tau(i+\tfrac{1}{2})\Phi'(i).
\end{array}
\right.
\end{equation}
Defining $\mu(i+\tfrac{1}{2})$ by equation \eqref{eq:DiscMu}, equation \eqref{eq:DiscMuLinha} still holds.
It is not clear how to define a discrete version of the intrinsic affine binormal developable.

\subsection{Polygons with $\mu$ constant}

Consider an equal area planar polygon $\Gamma(i+\tfrac{1}{2})$ and let $Z(i+\tfrac{1}{2})$ be given by 
\begin{equation*} \label{eq:DiscZ}
Z(i+\tfrac{1}{2})=\sum_{j=1}^{i} z(j),
\end{equation*}
where $z(i)$ is given by equation \eqref{eq:Discz}, for some point $P\in\R^2$. Then $Z(i+\tfrac{1}{2})$ represents the area of the planar region bounded by 
$\Gamma'(j)$, $j=1...i$, and the segments $P\Gamma(\tfrac{1}{2})$ and 
$P\Gamma(i+\tfrac{1}{2})$ (see Figure \ref{fig:AreaZ}). In this context, Proposition \ref{prop:Main} can be written as follows:

\begin{Proposition}
The polygon $\Phi=(\Gamma,Z)$ has constant $\mu$, and conversely, any equal-volume polygon $\Phi$ with constant $\mu$ 
is obtained by this construction, for some planar polygonal line $\Gamma$ and some point $P\in\R^2$.
\end{Proposition}

\begin{figure}[htb]
 \centering
 \includegraphics[width=0.45\linewidth]{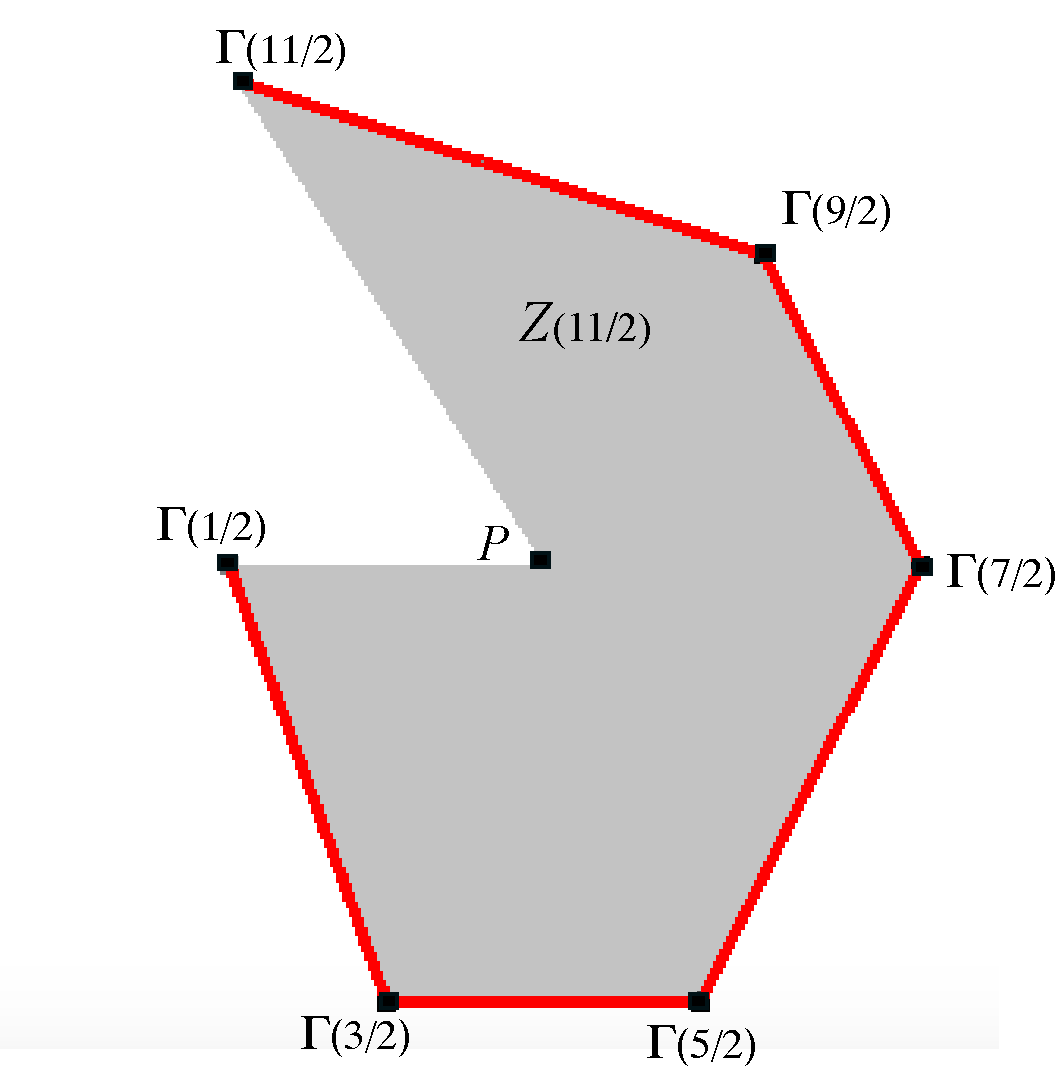}
 \caption{A planar equal-area polygon $\Gamma$ and the area \newline represented by $Z(\tfrac{11}{2})$. }
\label{fig:AreaZ}
\end{figure}

\section{Projective polygons}

In this section, we obtain discrete counterparts of the results of section \ref{sec:SmoothProjective}. 

Consider a planar polygon $\tilde\phi(i)$, $i=1,...,N$. Assume that
\begin{equation}\label{eq:DiscConvex}
\left[\tilde\phi'(i-\tfrac{1}{2}), \tilde\phi'(i+\tfrac{1}{2})  \right]=b(i)> 0.
\end{equation}

\subsection{Equal-volume representative}

Any polygon $\phi$ in $\R^3$ of the form $\phi(i)=a(i)\left(\tilde\phi(i),1\right)$, $a(i)> 0$, is a projective representative of $\tilde\phi$. 

\begin{lemma}
There exists a projective representative $\phi$ of $\tilde\phi$ such that equation \eqref{eq:EqualVolumeCA} holds with $O$ equal to the origin.
\end{lemma}
\begin{proof}
Observe first that 
$$
\left[ \phi(i-1), \phi(i), \phi(i+1)  \right]=a(i-1)a(i)a(i+1) \left[\tilde\phi'(i-\tfrac{1}{2}), \tilde\phi'(i+\tfrac{1}{2})  \right].
$$
So we need to choose $a(i)$, $i=1...,N$ such that 
\begin{equation}\label{eq:ab}
a(i-1)a(i)a(i+1)b(i)=c, \ \ i=2,...,N-1, 
\end{equation}
for some constant $c$. Since by the hypothesis \eqref{eq:DiscConvex} $b(i)> 0$, given $a(1)>0$ and $a(2)>0$
we can find unique $a(i)>0$, $i=3,...,N$ such that \eqref{eq:ab} holds.
\end{proof}

Assume that $\phi$ is a representative of $\tilde\phi$ such that equation \eqref{eq:EqualVolumeCA} holds with $O$ equal to the origin (Figure \ref{fig:Projective}). Then, by lemma \ref{lemma:Phi3LinhasTangente}, 
$\phi'''(i+\tfrac{1}{2})$ belongs to the plane generated by $\{\phi(i),\phi(i+1)\}$. So 
we can use equations \eqref{eq:DiscFrenetCA1} to define $\rho_1(i)$, $\rho_2(i)$ and $\tau(i+\tfrac{1}{2})$.

\begin{figure}[htb]
\centering
\subfigure{
\includegraphics[width=.45\textwidth]{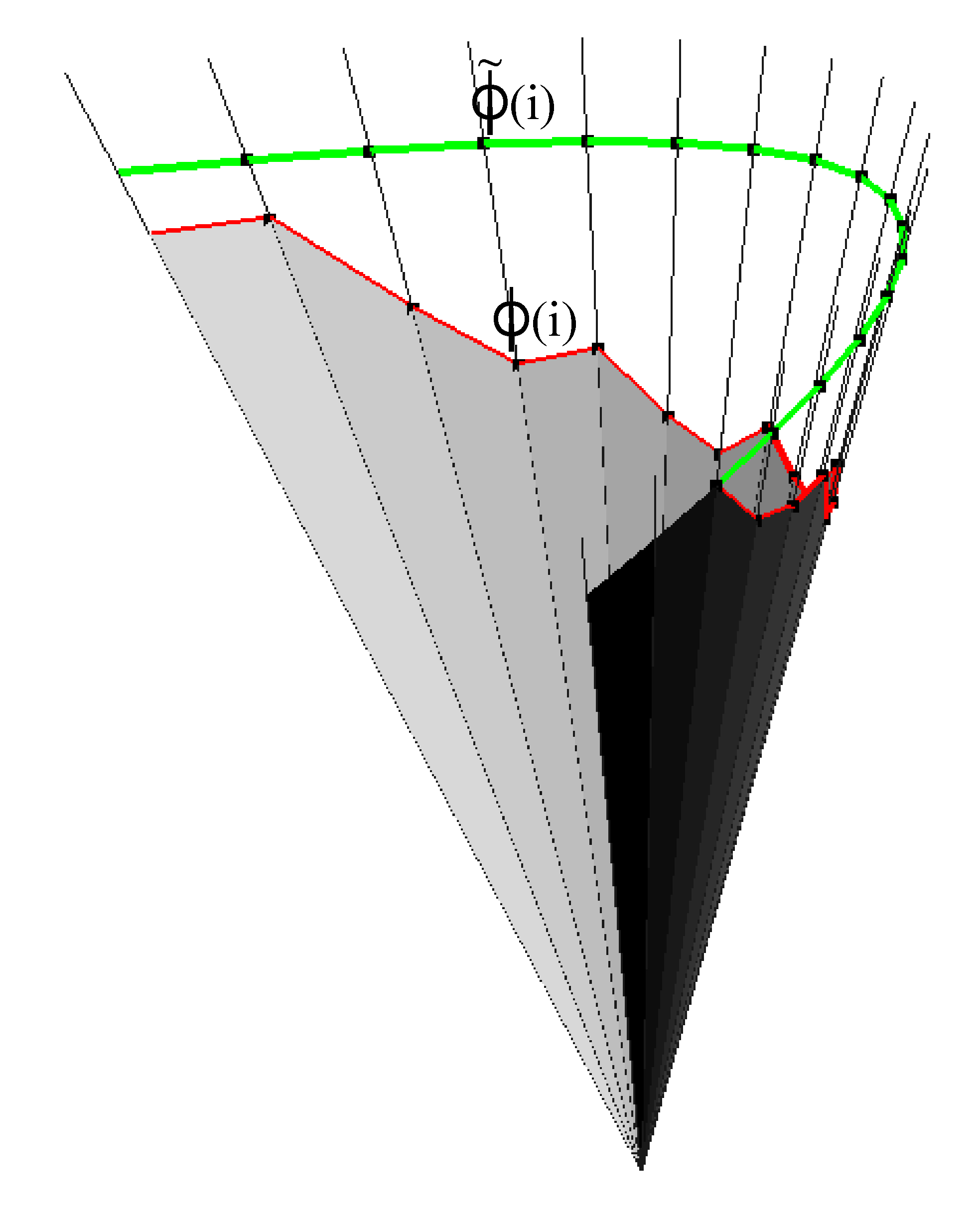}
}
\subfigure{
\includegraphics[width=.45\textwidth]{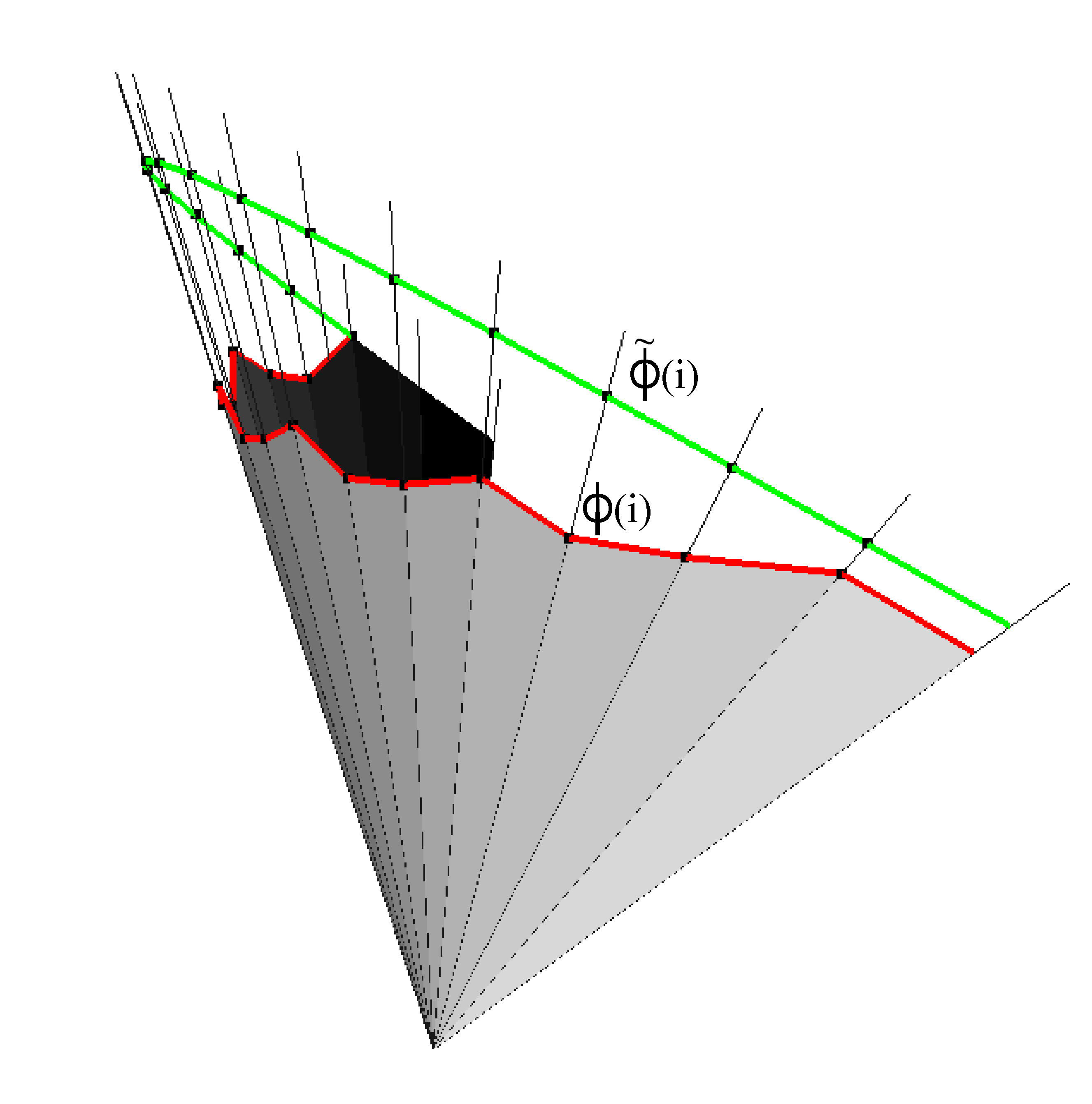}
}
\caption{Two views of a planar projective polygon $\tilde\phi$ and its equal-volume representative $\phi$.}
\label{fig:Projective}
\end{figure}

\subsection{Projective length}

We would like to define the projective length of $\tilde\phi$ as 
\begin{equation}\label{eq:DiscProjectiveLength1}
pl_1(\tilde\phi)=\sum_{j=2}^{N-2}  \left(  \rho_1'(i+\tfrac{1}{2})+2\tau(i+\tfrac{1}{2})  \right)^{1/3}
\end{equation}
or 
\begin{equation}\label{eq:DiscProjectiveLength2}
pl_2(\tilde\phi)=\sum_{j=2}^{N-2}  \left(  \rho_2'(i+\tfrac{1}{2})+2\tau(i+\tfrac{1}{2})  \right)^{1/3}
\end{equation}
but unfortunately these two definitions do not coincide. Nevertheless, if the polygonal line is obtained from a dense enough sampling of a smooth curve,
both of these formulas are close to projective length of the smooth curve given by equation \eqref{eq:ProjectiveLength}. Denote by $O(h^k)$
any quantity such that $\lim_{h\to 0}\frac{O(h^k)}{h^{k-\epsilon}}=0$, for any $\epsilon>0$.

\begin{lemma}\label{lemma:Convergence}
Assume that the polygonal line $\phi(i)$, $i=1,...N$, is obtained from $\phi(t)$, $0\leq t\leq T$, by uniform sampling. Then, for $Nh=T$, $ih=t$, we have 
$$
\rho_1'(i+\tfrac{1}{2})+2\tau(i+\tfrac{1}{2}) = \left(  \rho'(t)+2\tau(t) \right) h^3+ O(h^4).
$$
A similar result holds for $\rho_2$. 
\end{lemma}

\begin{proof}
It is standard in numerical analysis that $\phi'(i+\tfrac{1}{2})=h\phi'(t)+O(h^2)$ and  $\phi'''(i+\tfrac{1}{2})=\phi'''(t)h^3+O(h^4)$.
Thus equation \eqref{eq:FrenetCA} can be written as 
$$
\phi'''(i+\tfrac{1}{2})=-\rho(t)h^2\phi'(i+\tfrac{1}{2})+\tau(t)h^3\phi(i)+O(h^4).
$$
We conclude that $\tau(i+\tfrac{1}{2})=\tau(t)h^3+O(h^4)$ and $\rho_1(i+1)=\rho(t)h^2+c(t)h^3+O(h^4)$. This last equation implies that
$\rho_1'(i+\tfrac{1}{2})=h^3\rho'(t)+O(h^4)$. Thus we conclude that
$$
\rho_1'(i+\tfrac{1}{2})+2\tau(i+\tfrac{1}{2})=\left( \rho'(t)+2\tau(t) \right) h^3+O(h^4),
$$
which proves the lemma.
\end{proof}

From this lemma we can obtain the following convergence result:

\begin{corollary}
The discrete projective lengths given by equations \eqref{eq:DiscProjectiveLength1} and \eqref{eq:DiscProjectiveLength2} converge to 
the smooth projective length given by \eqref{eq:ProjectiveLength} when $h\to 0$.
\end{corollary}

\begin{example}
\label{ex:example}
Consider 
$$
\tilde\phi(t) = (\exp(-t)\cos(t),\exp(-t)\sin(t),1), \ \ \ \  0\leq t \leq 2\pi.
$$ 
Then $\tilde\phi$ is projectively equivalent to  $\phi(t)=2^{-1/3} \exp(2t/3)\tilde\phi(t)$, which satisfies equation \eqref{eq:CentroAffineArcLength} with $O$ equal the origin. Straightforward calculations
show that $\rho(t)=2/3$, $\tau(t)=20/27$ and 
$$
pl(\tilde\phi) = 2\pi \tfrac{\sqrt[3]{40}}{3} \approx 7.162519249.
$$  

We have done some experiments considering uniform samplings of this curve with $N$ points. Table \ref{tab:table1} presents the results for $N = 10,100,1000$. Observe that both $pl_1(\tilde\phi)$ and $pl_2(\tilde\phi)$ get closer to $pl(\tilde\phi)$ as $h=\frac{2\pi}{N}$ decreases.   
\begin{table}[htb]
\centering
\caption{Experimental results of example \ref{ex:example}.}
\label{tab:table1}
\begin{tabular}{ c  c  c  c  c }

N & h & $pl_1(\tilde\phi)$ & $pl_2(\tilde\phi)$  &  \\ \hline \hline

10      & 0.62831 & 4.26627 & 3.55522 &   \\ \hline
100    & 0.06283 & 6.87572 & 6.80410 &  \\ \hline
 1000 & 0.00628 & 7.13407  & 7.12691  &  \\ \hline

\end{tabular}
\end{table}
\end{example}

\end{document}